\documentclass[11pt]{amsart}
\usepackage[a4paper]{geometry}
\usepackage{times}
\usepackage[british]{babel}
\usepackage[utf8]{inputenc}
\usepackage[pdftex]{color, graphicx}
\usepackage{amsmath,amsthm}
\usepackage{amssymb,amsfonts}
\usepackage[foot]{amsaddr}
\usepackage[hang,flushmargin]{footmisc} 
\usepackage{mathrsfs}
\usepackage{graphicx}
\usepackage{enumerate}
\usepackage{xypic}
\usepackage{microtype}
\usepackage{hyperref}
\usepackage[capitalize]{cleveref}
\usepackage{amscd}
\usepackage{booktabs}
\renewcommand{\star}{{}^\ast}
\usepackage{mathtools}
\DeclarePairedDelimiter{\set}{\{}{\}}
\DeclarePairedDelimiter{\abs}{\lvert}{\rvert}

\DeclareMathOperator{\dcl}{dcl}

\usepackage{centernot}

\usepackage[autostyle, italian=guillemets]{csquotes}
\usepackage{enumitem}
\usepackage{upgreek}

\newcommand{\tp}{\textrm{tp}}
\newcommand{\tmid}{\mathrel{\tilde \mid}}
\newcommand{\from}{\colon}
\renewcommand{\phi}{\varphi}
\newcommand{\restr}{\upharpoonright}
\newcommand{\proves}{\mathrel{\vdash}}
\newcommand{\bla}[4]{{#1}_{#2}#3\ldots#3{#1}_{#4}}

\newcommand{\monster}{\mathfrak U}

\theoremstyle{definition}
\newtheorem{theorem}{Theorem}[section]
\newtheorem{lemma}[theorem]{Lemma}

\newtheorem{corollary}[theorem]{Corollary}

\newtheorem{proposition}[theorem]{Proposition}

\newtheorem{defin}[theorem]{Definition}

\newtheorem{eg}[theorem]{Example}
\newtheorem{problem}[theorem]{Problem}

\newtheorem{rem}[theorem]{Remark}
\newtheorem{ass}[theorem]{Assumption}
\newtheorem*{mainthm}{Main Theorem}
\newtheorem*{mainco}{Main Corollary}

\usepackage{fancyhdr}
\usepackage{orcidlink}

\usepackage{url}
\makeatletter                        
\g@addto@macro{\UrlBreaks}{\UrlOrds} 
\makeatother

\makeatletter
\def\paragraph{\@startsection{paragraph}{4}%
  \z@\z@{-\fontdimen2\font}%
  {\normalfont\bfseries}}
\makeatother

\title{Extending orders to types}

\author{Lorenzo Luperi Baglini$^*$\,\orcidlink{0000-0002-0559-0770}}
\address{$^*$ Dipartimento di Matematica, Università  di Milano, Via Saldini 50, 20133 Milano, Italy}
\author{Marcello Mamino$^\dagger$\,\orcidlink{0000-0001-6103-9775}}
\address{$^\dagger$ Dipartimento di Matematica, Universit\`a di Pisa, Largo Bruno Pontecorvo 5, 56127 Pisa, Italy}
\author{Rosario Mennuni\,\orcidlink{0000-0003-2282-680X}}
\author{Mariaclara Ragosta$^\ddagger$\,\orcidlink{0009-0004-6641-4676}}
\address{$^\ddagger$ Department of Applied Mathematics (KAM), Charles University, Malostranské náměstí 25, Praha 1, Czech
Republic}
\author{Boris \v Sobot$^\dagger{}^\dagger$\,\orcidlink{0000-0002-4848-0678}}
\address{$^\dagger{}^\dagger$ Faculty of Sciences, University of Novi Sad, Trg Dositeja Obradovi\' ca 4, 21000 Novi Sad, Serbia}


\begin{document}

\begin{abstract}
  Given an ordered structure, we study a natural way to extend the order to preorders on type spaces. For definably complete, linearly ordered structures, we give a characterisation of the preorder on the space of 1-types. We apply these results to the divisibility preorder on the space of ultrafilters on the set of natural numbers, giving an independence result about the suborder consisting of ultrafilters with only one fixed prime divisor, as well as a classification of ultrafilters with finitely many prime divisors.
\end{abstract}

\maketitle

\section{Introduction}

Let $(M,\le,\ldots)$ be an expansion of an infinite partial order. In this note we study a construction extending $\le$ to preorders on type spaces, mainly focusing on the case where $M$ is a definably complete linear order.

More in detail, for $A$ a set of parameters in a partially ordered structure $(M,\le, \ldots)$, and $k$ a natural number, we equip $S_k(A)$ with a preorder,  denoted by $\precapprox$, that is defined as follows. For two $k$-types $p$ and $q$, we write $p\precapprox q$ iff, for some realisations $\alpha\models p$ and $\beta\models q$ in an elementary extension $\monster$, we have $\alpha\le \beta$ in the induced product (partial) order on $\monster^k$. 

Equipping type spaces with relations is not a new idea. In particular, the definition we will use is equivalent to one that has already been considered in the past for arbitrary relations, see for instance Definition~1.3 in~\cite{poliakovUltrafilterExtensionsFirstorder2021} and references therein. Another natural preorder, that we will not study here, may be defined by using tensor products,  see~\cite{savelievUltrafilterExtensionsLinearly2015}.  In a different direction, a certain relational structure on type spaces is key to Hrushovski's \emph{definability patterns}~\cite{hrushovskiDefinabilityPatternsTheir2020}.

The properties of the relation $\precapprox$, at the general level of partially ordered structures, are studied in \Cref{sec:posets}. In \Cref{sec:defcomplin} we study the posets $S_k(A)/\mathord\approx$, where $\approx$ is the equivalence relation induced by $\precapprox$, in the case of linearly ordered, definably complete structures (\Cref{def:defcmpllo}).  Definably complete dense linear orders were introduced in \cite{millerExpansionsDenseLinear2001}; under these assumptions, analogues of many theorems from real analysis hold~\cite{fornasieroDefinablyCompleteBaire2010,hieronymiAnalogueBaireCategory2013}. Examples of (not necessarily dense) definably complete linear orders include every expansion of $(\mathbb N,\le)$, or of $(\mathbb R, \le)$, and every o-minimal structure. Our main result is the following characterisation.
\begin{mainthm}[\Cref{thm:approx1char}]
  Let $(M,\le,\ldots)$ be a definably complete expansion of a linear order and, for $A\subseteq M$, let $\operatorname{CC}(A)$ be the set of cuts in the definable closure of $A$ that are filled in some elementary extension, with its natural linear order. Then ${(S_1(A)/\mathord\approx)}\cong \operatorname{CC}(A)$.
\end{mainthm}

The original motivation for this work comes from the study of divisibility and congruences on the space $\upbeta \mathbb N$ of ultrafilters on the set of natural numbers, see, e.g., \cite{sobotCONGRUENCEULTRAFILTERS2021,dinassoSelfdivisibleUltrafiltersCongruences2023a}. When the starting partial order is divisibility of natural numbers, and $\mathbb N$ is viewed as a structure with a symbol for every subset of each $\mathbb N^k$, the relations $\precapprox$, $\approx$ induced on $S_1(\mathbb N)\cong \upbeta \mathbb N$ are denoted by $\tmid$ and $=_\sim$ respectively. If $p$ is prime (\Cref{defin:primeuf}), we write $\mathcal E_p$ for the restriction of the poset $(\upbeta \mathbb N/\mathord =_\sim, \,\tmid\,)$ to equivalence classes of ultrafilters of the form $\tp(\gamma^\delta/\mathbb N)$, for $\gamma\models p$ and $\delta$ some nonstandard integer. In \Cref{sec:div} we see that the study of $\mathcal E_p$ naturally corresponds to studying $\precapprox$ on $S_1(\gamma)$ (\Cref{rem:tmidtoprecapprox}), and deduce the following.

\begin{mainco}[\Cref{thm:Epchar,thm:isocpindep}]
  For every prime ultrafilter $p$, if $\gamma\models p$ then the order $\mathcal E_p$ is isomorphic to $\operatorname{CC}(\gamma)$. It is independent of $\mathrm{ZFC}$ whether all $\mathcal E_p$ with $p$ nonprincipal are isomorphic.
\end{mainco}
If $q$ is the type of an increasing $k$-tuple of primes, one can similarly look at the suborder $\mathcal E_{q}$, whose study reduces to the study of $\precapprox$ on $S_k(\gamma)$. Although, for $k\ge 2$, a structure theorem for $\mathcal E_q$ is still lacking, we provide a classification of its points.

We conclude by leaving some open questions in \Cref{sec:croq}.

\medskip

\paragraph{Acknowledgements}

We thank Mauro Di Nasso for useful conversations and the anonymous referee for their thorough reading, their useful comments, and for catching a mistaken statement in a previous version of this paper. Much of this research was conducted in the coffee room of the Department of Mathematics of the University of Pisa; we thank the Department for its hospitality.

\medskip

\paragraph{Funding}

 LLB, RM and MR were supported by the project PRIN 2022 ``Logical methods in combinatorics'', 2022BXH4R5, Italian Ministry of University and Research (MUR). MM and RM were supported by the MUR project PRIN 2022: ``Models, sets and classifications'' Prot.~2022TECZJA. We acknowledge the MUR Excellence Department Project awarded to the Department of Mathematics, University of Pisa, CUP I57G22000700001.  RM is a member of the INdAM research group GNSAGA. MR is supported by project 25-15571S of the Czech Science Foundation (GAČR). B\v S was supported by the Science Fund of the Republic of Serbia (call IDEJE, project SMART, grant no.\ 7750027) and by the Ministry of Science, Technological Development and Innovation of the Republic of Serbia (grant no.\ 451-03-47/2023-01/200125).

\section{Types in ordered structures}\label{sec:posets}
Let $(M,\le,\ldots)$ be an infinite structure expanding a partial order. We also denote by $\le$ the induced product order on its Cartesian powers $M^k$; namely, $(\bla a1,k)\le (\bla b1,k)$ iff for every $i\le k$ we have $a_i\le b_i$.

We adopt standard model-theoretic conventions. E.g., we fix a sufficiently large cardinal $\kappa$, and work inside a $\kappa$-saturated and $\kappa$-strongly homogeneous monster model  $\monster$.  \emph{Small} means ``of cardinality less than $\kappa$''. Parameter sets, usually denoted by $A$,  are assumed to be small subsets of $\monster$, unless otherwise indicated. The relation of having the same type over $A$ is denoted by $\equiv_A$.  Realisations of types over small sets are assumed to be in $\monster$. We sometimes write $x$ for a tuple of variables $x=(\bla x1,k)$, and similarly for tuples of elements. \emph{Definable} means ``definable with parameters''; similarly, formulas are allowed to use parameters from sets that will be clear from context. By \emph{type} we will mean ``complete type over some parameter set'' but, as usual,   \emph{type-definable} means ``definable by a partial type''. If $p\in S_k(A)$ and $f$ is an $A$-definable function, by $f_*p$ we denote the pushforward of $p$ along $f$. We write $\dcl$ for the definable closure operator.

\subsection{Ordering types}

\begin{defin}\label{def:lessapprox}
Let $A\subseteq M$ and $k\in \mathbb N$. For $p,q\in S_k(A)$, we define $p\lessapprox q$ iff there are $\alpha\models p$ and $\beta\models q$ such that $\alpha\le \beta$. We write $p\approx q$ iff $p\lessapprox q\lessapprox p$.
\end{defin}
\begin{rem}\label{rem:freeforall}
 By an automorphism argument, $p\lessapprox q$ if and only if for every $\alpha\models p$ there is $\beta\models q$ such that $\alpha\le \beta$, if and only if for every $\beta\models q$  there is $\alpha\models p$ such that $\alpha\le \beta$.
\end{rem}
\begin{rem}
  Clearly, $S_k(A)/\mathord\approx$ is partially ordered by the relation induced by $\precapprox$, which we also denote by the same symbol. Whenever we mention $S_k(A)/\mathord\approx$, we consider it equipped with this partial order.
\end{rem}

 We call a formula $\phi(\bla x1,k)$ \emph{upward closed} [respectively, \emph{downward closed}] iff the set $\phi(M)\subseteq M^k$ it defines is upward closed [respectively, downward closed].

\begin{proposition}\label{pr:upwclosedchar}
For every $p,q\in S_k(A)$, the following are equivalent. 
\begin{enumerate}
\item \label{point:pleq} $p\lessapprox q$.
\item \label{point:upwardclosed}If $\phi(x)\in p(x)$ is upward closed, then $\phi(x)\in q(x)$.
\item \label{point:downwardclosed}If $\phi(x)\in q(x)$ is downward closed, then $\phi(x)\in p(x)$.
\end{enumerate}
All conditions above also imply the following, and are equivalent to it provided that $k=1$ or that $\dcl(A)$ contains two comparable points.
\begin{enumerate}[resume*]
\item \label{point:incrfun}For all $A$-definable, increasing partial functions $f\from M^k\to M$ such that both $p$ and $q$ contain the domain of $f$, we have  $f_* p\lessapprox f_* q$.
\end{enumerate}
\end{proposition}
\begin{proof}

     $\eqref{point:pleq}\Rightarrow \eqref{point:upwardclosed}$: immediate from the definitions.

  $\eqref{point:upwardclosed}\Rightarrow \eqref{point:pleq}$: suppose that $p\centernot \lessapprox q$. Then, by compactness, there are $\theta(x)\in p(x)$ and $\psi(y)\in q(y)$ such that $\theta(x)\land \psi(y)\proves x\centernot \le y$. Let $\phi(y)$ be the upward closure of $\theta(y)$, namely, $\phi(y)\coloneqq  \exists z\; (\theta(z)\land z\le y)$. Then, $\phi(y)$ is inconsistent with $\psi(y)$, hence cannot belong to $q(y)$. But $\theta(x)$ implies $\phi(x)$, hence the latter belongs to $p(x)$. 

$\eqref{point:upwardclosed}\Leftrightarrow\eqref{point:downwardclosed}$: follows from the observation that the complements of upward closed sets are precisely the downward closed sets.

$\eqref{point:pleq}\Rightarrow \eqref{point:incrfun}$: immediate from the definitions.

$\eqref{point:incrfun}\Rightarrow\eqref{point:pleq}$ in the case $k=1$: apply $\eqref{point:incrfun}$ to the identity function.

$\eqref{point:incrfun}\Rightarrow \eqref{point:upwardclosed}$ if $\abs {\dcl(A)}\ge 2$: fix two elements $0<1\in \dcl(A)$, and use them to write characteristic functions of definable sets, with value $1$ on points belonging to the set. For upward closed sets such characteristic functions are increasing.
\end{proof}
We will say that \emph{$p$ lies on an antichain} to mean that some definable set in $p$ is an antichain.
\begin{proposition}\label{pr:charac}
  Let $p\in S_k(A)$.  The following are equivalent.
  \begin{enumerate}
  \item \label{point:increl} There are no $\alpha\ne\alpha'\models p$ such that $\alpha\le \alpha'$.
  \item \label{point:antichain} $p$ lies on an antichain.
  \end{enumerate}
\end{proposition}
\begin{proof}
  If $\eqref{point:increl}$ holds, by compactness and up to taking conjunctions there is $\phi(x)\in p(x)$ such that $\phi(x)\land \phi(y)\proves x=y\lor x\centernot\le y$. Hence $\phi(x)$ defines an antichain, proving $\eqref{point:increl}\Rightarrow \eqref{point:antichain}$, and $\eqref{point:antichain}\Rightarrow \eqref{point:increl}$ is obvious.
\end{proof}

\begin{proposition}\label{pr:singleton}
For every $p\in S_k(A)$, the following are equivalent.
  \begin{enumerate}
   \item \label{point:singleton} The $\approx$-class of $p$ is a singleton.
   \item \label{point:singletoneq} The locus of $p$ is convex, that is, if $\alpha,\alpha'\models p$ and $\alpha\le\beta\le\alpha'$ then $\beta\models p$.
  \end{enumerate}
\end{proposition}
\begin{proof}
For $\eqref{point:singleton}\Rightarrow \eqref{point:singletoneq}$, if there are $\alpha,\alpha'\models p$ and $\beta$ is such that  $\alpha\le\beta\le\alpha'$, then $\tp(\beta/A)\approx p$ by definition, hence by \eqref{point:singleton} we must have $\beta\models p$. For $\eqref{point:singletoneq}\Rightarrow \eqref{point:singleton}$, assume that $p\approx q$. By \Cref{rem:freeforall} there are $\alpha,\alpha'\models p$ and $\beta\models q$ such that $\alpha\le \beta\le \alpha'$. It then follows from \eqref{point:singletoneq} that $p=q$.
\end{proof}

\begin{proposition}\label{pr:antichsing}
  If $p$ lies on an antichain, then its $\approx$-class is a singleton.
\end{proposition}
\begin{proof}
Point \eqref{point:increl} of \Cref{pr:charac} clearly implies point \eqref{point:singletoneq} of \Cref{pr:singleton}.
\end{proof}

It can happen that the $\approx$-class of a type $p$ is a singleton even if $p$ contains no antichain.

\begin{eg}\label{ex:Q}
In $(\mathbb Q,<)$, or more generally in an arbitrary o-minimal structure, every $1$-type lies in a singleton $\approx$-class, but linearity of the order implies that there are no antichains with more than one point.
\end{eg}

\begin{rem}\label{rem:scfsloa}
  Later we will deal with a structure where, for every type $p$, and every pair $\alpha<\alpha'$ of realisations of $p$, there is $\beta\centernot\models p$ with $\alpha\le\beta\le\alpha'$. By \Cref{pr:singleton,,pr:antichsing,pr:charac}, this condition implies that the $\approx$-class of $p$ is a singleton if and only if $p$ lies on an antichain.
\end{rem}

It may happen that some interval $[\alpha,\alpha']$ consists entirely of realisations of a type $p$, but the locus of $p$ is not convex. This is visible in the following example.

\begin{eg}\label{ex:qwpwie}
 Consider the structure $M=(\mathbb Q,\le, X)$, where $X=\bigcup_{m\in \mathbb Z}\{a\in\mathbb{Q}:2m+\pi<a<2m+1+\pi\}$. An easy back-and-forth argument shows that the theory of this structure eliminates quantifiers after naming, for each $n\in \omega$, the  relation $R_n(x,y)$ that holds iff the set  $\set{k\in \mathbb Z : x<k + \pi < y}$ has size $n$.

  Note that $R_0$ is an equivalence relation, and each equivalence class is either a maximal convex subset of $X$ or a maximal convex subset of $\mathbb Q\setminus  X$. Moreover, $R_n(x,y)$ counts how many equivalence classes meet $[x,y]$. 

It follows from quantifier elimination that there are precisely two elements of $S_1(M)$ that are larger than $M$, one in $X$, call it $p$, and one in its complement, call it $q$. If $\alpha,\alpha'\models p$ and $\models R_0(\alpha,\alpha')$, then $[\alpha,\alpha']$ consists entirely of realisations of $p$. Yet, the locus of $p$ is not convex, since if $\alpha,\alpha'\models p$ and $\models R_2(\alpha,\alpha')$ then there is $\beta\models q$ with $\alpha<\beta<\alpha'$.
\end{eg}

     \begin{proposition}\label{pr:box}For $p\in S_k(A)$, the following are equivalent.
 \begin{enumerate}
   \item\label{point:box1} There are $\alpha,\alpha'\models p$ such that, for every $i\le k$, we have $\alpha_i<\alpha_i'$.
   \item\label{point:box2} Let $\alpha\models p$. For $i\le k$, there are infinite intervals $I_i=(\gamma_i, \delta_i)$, definable over $\monster$, containing $\alpha_i$, with both $(\gamma_i, \alpha_i)$ and $(\alpha_i, \delta_i)$ infinite, and such that, for all   $\beta\in \bla I1\times k$, we have $\tp(\beta/A)\approx p$.
   \end{enumerate}
    \end{proposition}
    \begin{proof}$\eqref{point:box1}\Rightarrow\eqref{point:box2}$: by an automorphism argument, given an arbitrary $\alpha\models p$, we can find $\alpha'$ as in $\eqref{point:box1}$. By this and compactness, we can then find a sequence $(\alpha^j : j\in \mathbb Z\cup \set{-\infty, +\infty})$ of realisations of $p$ with $\alpha^0=\alpha$ and such that if $j<\ell$ then $\alpha^j_i<\alpha^\ell_i$. It now suffices to set $I_i\coloneqq (\alpha^{-\infty}_i, \alpha^{+\infty}_i)$.

      $\eqref{point:box2}\Rightarrow\eqref{point:box1}$: fix $\alpha\models p$ and let the $I_i$ be given by the assumption. Let $\beta\in \bla I1\times k$ be such that, for every $i\le k$, we have  $\alpha_i < \beta_i$. As $\tp(\beta/A)\approx p$, there is $\alpha'\models p$ such that $\beta \le \alpha'$, and for all $i\le k$ we have $\alpha_i<\beta_i\le \alpha_i'$.
    \end{proof}

\subsection{Linear orders}
\begin{rem}
  If $(M,\le)$ is a linear order, then
  \begin{enumerate}[label=(\alph*)]
  \item $S_1(A)/\mathord\approx$ is linear, and
  \item if $\dcl(A)\ne \emptyset$, then $S_k(A)/\mathord\approx$ is linear if and only if $k=1$. In fact, if $k\ge 2$, $a\in\dcl(A)$, $b\le a\le c$, and $b\ne c$, then the types of $(b,b,\ldots, b,c)$ and $(c,c,\ldots, c, b)$ over $A$ are incomparable.
  \end{enumerate}
\end{rem}

\begin{proposition}\label{pr:antichar} Let $(M,<)$ be  a linear order, $D\subseteq M^k$ and, for $i\in \set{1,\ldots,k}$,
  \[  G_i\coloneqq \{(x_1,\ldots,x_{i-1},x_{i+1},\ldots,x_k,x_i):(x_1,x_2,\ldots,x_k)\in D\}.
  \]
 The following conditions are equivalent.
 \begin{enumerate}
   \item \label{point:decfun1} $D$ is an antichain.
   \item \label{point:decfun2} For every $i\in\{1,2,\ldots,k\}$, the set  $G_i$
 is the graph of a partial strictly decreasing function.
\item \label{point:decfun3} For some $i\in\{1,2,\ldots,k\}$, the set $G_i$ is the graph of a partial strictly decreasing function.
   \end{enumerate}
\end{proposition}  
\begin{proof}
For $\eqref{point:decfun1}\Rightarrow\eqref{point:decfun2}$, fix a coordinate, without loss of generality $x_k$. If $(\bla a1,k)\in D$ and $a\ne a_k$, by linearity $a< a_k$ or $a>a_k$, hence $(\bla a1,{k-1},a)\notin D$ and $D$ can be regarded as the graph of a partial function of $\bla x1,{k-1}$, call it $f$. Let us check that $f$ is strictly decreasing. If $(\bla a1,{k-1})<(\bla b1,{k-1})$, as $D$ is an antichain, we must have $(\bla a1,{k-1},f(\bla a1,{k-1}))\centernot \le (\bla b1,{k-1}, f(\bla b1,{k-1}))$, and since the order on $M$ is linear it follows that $f(\bla a1,{k-1}) > f(\bla b1,{k-1})$. The implication $\eqref{point:decfun2}\Rightarrow\eqref{point:decfun3}$ is obvious, and the proof of $\eqref{point:decfun3}\Rightarrow\eqref{point:decfun1}$ is easy.
\end{proof}

\begin{corollary}\label{cor:liesanti}
If $p$ is a $k$-type in a linear order, the following conditions are equivalent.
 \begin{enumerate}
   \item\label{point:acgraphtype1} $p$ lies on an antichain.
   \item\label{point:acgraphtype2} For some/every $i\in\{1,2,\ldots,k\}$, the set
   \[\{(x_1,\ldots,x_{i-1},x_{i+1},\ldots,x_k,x_i):(x_1,x_2,\ldots,x_k)\models p\}\]
   is the graph of a partial strictly decreasing function.
   \end{enumerate}
 \end{corollary}
 \begin{proof}
   $\eqref{point:acgraphtype1}\Rightarrow \eqref{point:acgraphtype2}$: 
   if $p$ lies on an antichain, its realisations form an antichain and we apply \Cref{pr:antichar}.

   $\eqref{point:acgraphtype2}\Rightarrow \eqref{point:acgraphtype1}$: by \Cref{pr:antichar} the set of realisations of $p$ is an antichain, and we apply compactness.
\end{proof}

\section{Definably complete structures}\label{sec:defcomplin}
\begin{defin}\label{def:defcmpllo}
  A linearly ordered structure $M$ is \emph{definably complete} iff, for every nonempty definable subset $X$ of $M$, if $X$ has an upper bound then it has a supremum, and if $X$ has a lower bound then it has an infimum.
\end{defin}
\begin{rem}\label{rem:defcmpl}
  \begin{enumerate}[label=(\alph*),leftmargin=*]
  \item\label{point:dcmplth} Whether $M$ is definably complete only depends on its theory; that is, if $M\equiv N$, then $M$ is definably complete if and only if $N$ is.         
  \end{enumerate}
  \begin{enumerate}[label=(\alph*),resume*]
\item\label{point:infindcl} If $X\subseteq M$ is $A$-definable and $\inf(X)$ exists, then $\inf(X)\in \dcl(A)$.
  \end{enumerate}
\end{rem}

\begin{ass}
For the rest of the section, unless otherwise stated, we assume $(M,\le,\ldots)$ to be a linearly ordered,  definably complete structure.
\end{ass}

\begin{defin}
  \begin{enumerate}[label=(\alph*),leftmargin=*]
  \item A \emph{cut} in a linear order $(X,\le)$ is a pair $(L,R)$ of subsets of $X$ such that $L\cup R=X$ and for all $\ell\in L$ and $r\in R$ we have $\ell\le r$.
  \end{enumerate}
    \begin{enumerate}[label=(\alph*),resume*]
\item If $(X,\le)\subseteq (Y,\le)$ are linear orders and $(L,R)$ is a cut in $X$, we say that $b\in Y$ \emph{realises} $(L,R)$ iff $L<b<R$ or $\set{b}=L\cap R$.
  \end{enumerate}
\end{defin}
Observe that one of $L,R$ may be empty, and that $L$ and $R$ are either disjoint or intersect in a unique point. Cuts of the second kind are precisely those that are realised by exactly one point in every linear order in which $X$ embeds. Clearly, the set of realisations of a given cut is always a convex set.

\begin{defin}
For $p\in S_1(A)$, we define $L_p\coloneqq \set{a\in \dcl(A): p(x)\proves x\ge a}$ and  $R_p\coloneqq \set{a\in \dcl(A): p(x)\proves x\le a}$. The \emph{cut} of $p$ is the type-definable set $\set{x=a}$ if $L_p\cap R_p=\set a$, or the type-definable set $\set{x> a : a\in L_p}\cup \set{x< a: a\in R_p}$ if $L_p\cap R_p=\emptyset$.
\end{defin}

\begin{lemma}\label{lemma:bicofinalincut}
Let $p(x)\in S_1(A)$, and let $C\subseteq \monster$ be the set of realisations of its cut. Then, $p$ is realised coinitially and cofinally many times in $C$.
\end{lemma}
\begin{proof}
If $p$ is realised in $\dcl(A)$ then $C$ is a singleton and the conclusion is trivial, hence we may assume it is not. In particular, because in linearly ordered structures algebraic closure coincides with definable closure, $p$ must have infinitely many realisations.
  
  Observe that every realisation of $p$ in $\monster$ must be in $C$.  It suffices to show that, for every $\beta\in C$, there are realisations of $p$ to its left and to its right. We prove the first statement; the second one follows by considering the structure with the reverse order.

  It is enough to prove that $p(x)\cup \set{x<\beta}$ is consistent. By compactness, it suffices to show that if $\phi(x)$ is a formula in $p(x)$, then $\phi(x)\land (x<\beta)$ is consistent. Since $\phi(x)$ belongs to a type, it defines a nonempty set.

  If $L_p\ne \emptyset$, up to taking a conjunction with a formula in $p(x)$ bounded from below, by definable completeness we may assume that $\phi(\monster)$ has an infimum $a$, which belongs to $\dcl(A)$ by \Cref{rem:defcmpl}\ref{point:infindcl}. Because $\phi(x)\in p(x)$, we must have $a\in L_p$, for otherwise $p$ would be inconsistent. In particular,  $a<\beta$, hence in $\monster$, by definition of infimum, there is a point of $\phi(\monster)$ in the interval $[a,\beta)$.

 If instead $L_p=\emptyset$, then we are assuming that $p$ is the type of a point smaller than $\dcl(A)$ and what we have to show is that $\phi(\monster)$ has arbitrarily small points. If not, by definable completeness it has an infimum in $\dcl(A)$. But then $\phi(x)$ cannot be consistent with being smaller than $\dcl(A)$.
\end{proof}
\begin{eg}
  The assumption of definable completeness is necessary. For instance, if $M$ is the expansion of $(\mathbb Q,\le)$ by a predicate $P$ for $\set{x\in \mathbb Q : x< \pi}$, then the type $p\in S_1(M)$ defined as $p(x)=\set{P(x)}\cup \set{a<x : a\in P(\mathbb Q)}$ is not realised cofinally in its cut.
\end{eg}
\begin{defin}
The space $\operatorname{CC}(A)$ of \emph{consistent cuts} over $A$ is the space of quantifier-free $1$-types in the language $\set{\le}$ over $\dcl(A)$ (where $\dcl$ is computed in the original structure, not in the reduct to $\set\le$). We view it as a linear order in the natural way.
\end{defin}
\begin{rem}\label{rem:qftypesarelegalcuts}
  If $A$ is a small set, $\operatorname{CC}(A)$ can be identified with the set of cuts in $\dcl(A)$ which are realised by some element of $\monster$.  Equivalently, this is the space of ultrafilters on the Boolean algebra of subsets of $\monster$  generated by intervals with endpoints in $\dcl(A)$.
\end{rem}

\begin{theorem}\label{thm:approx1char}
The natural map $S_1(A)/\mathord\approx \to \operatorname{CC}(A)$, sending an $\approx$-class to its cut, is an isomorphism of linear orders.
\end{theorem}
\begin{proof}

If the cut of $p$ lies below the cut of $q$, by definition there is $a\in \dcl(A)$ such that  $p(x)\proves x<a$ and $q(x)\proves x\ge a$, or such that $p(x)\proves x\le a$ and $q(x)\proves x> a$. Hence, every realisation of $p$ lies below every realisation of $q$, witnessing that $p\centernot \succapprox q$, and in particular $p\centernot \approx q$. It follows that the map $\pi\from S_1(A)/\mathord\approx \to \operatorname{CC}(A)$  sending the class of a $1$-type $p$ over $A$ to the cut of $p$ in $\dcl(A)$ is well-defined and increasing. Moreover, it is injective by \Cref{lemma:bicofinalincut}, and surjective by \Cref{rem:qftypesarelegalcuts}.
    \end{proof}
    \begin{eg}
      The analogue of \Cref{thm:approx1char} in higher dimension, with $S_k(A)$ in place of $S_1(A)$ and with  $\operatorname{CC}(A)$ replaced by the space of quantifier-free $k$-types in the language $\set{\le}$ over $\dcl(A)$, is false.  For instance, in the full theory on $\mathbb N$ (see beginning of \Cref{sec:div}), let $\gamma$ be an infinite point and work over the prime model $\mathbb N(\gamma)$ over $\gamma$. Let $\alpha$ be a realisation of the cut just after $\mathbb N$, that is, the cut $(\mathbb{N},\dcl(\{\gamma\})\setminus\mathbb{N})$. Then  $\tp((\alpha,\gamma-\alpha)/\mathbb N(\gamma))$ and  $\tp((\alpha,\gamma-\alpha-1)/\mathbb N(\gamma))$ have the same quantifier-free part in the language $\set{\le}$, but $(\alpha,\gamma-\alpha)$ lies on the antichain $\{(x,y): x+y=\gamma\}$, hence the $\approx$-class of its type is a singleton by \Cref{pr:antichsing}.
    \end{eg}

    By \Cref{pr:upwclosedchar}, if $\abs{\dcl(A)}\ge 2$  then one can characterise $\approx$-classes by using increasing $A$-definable partial functions $\monster^k\to \monster$. We leave it open (\Cref{prob:descska}) to give a more explicit description of the poset $S_k(A)/\mathord\approx$. 

    \section{Divisibility by finitely many primes}\label{sec:div}
    We now come to the original motivation of this work, namely, divisibility between ultrafilters. View  $\mathbb N$ as a structure in the richest possible language\footnote{Up to having the same definable sets.}, namely the one containing, for every $k$, a symbol for every relation or function on $\mathbb N^k$. Observe that ultrafilters in $\upbeta \mathbb N^k$ are naturally identified with $k$-types over $\mathbb N$.

    Write $\star \mathbb N$ (instead of $\monster$) for a monster model of this theory. By \Cref{rem:defcmpl}\ref{point:dcmplth} and the fact that $\mathbb N$ is (obviously) definably complete with respect to its natural order, so is $\star \mathbb N$. As the language at hand trivially has built-in Skolem functions, the definable closure of every set $A\subseteq \star \mathbb N$ is a model, the \emph{prime model} $\mathbb N(A)$ over $A$. For $\gamma$ a single element, write $\mathbb N(\gamma)$ instead of $\mathbb N(\set\gamma)$.  Every point of $\mathbb N(\gamma)$ is of the form $f(\gamma)$ for a suitable $f\from \mathbb N\to \mathbb N$.
    \begin{rem}\label{rem:dclultrapower}
      Up to isomorphism, $\mathbb N(\gamma)$ is the ultrapower of $\mathbb N$ along the ultrafilter corresponding to $\tp(\gamma/\mathbb N)$. In nonstandard-analytical parlance, we say that $\gamma$ is a \emph{generator} of this ultrafilter.
    \end{rem}
    \begin{rem}\label{rem:eqinfdist}
      Distinct tuples with the same type over an arbitrary parameter set must have infinite distance, since if $\gamma$ and $\delta$ differ on the coordinate $i$ and $\abs{\gamma_i-\delta_i}=n$ then the remainder class of $\gamma_i$ modulo $n+1$ differs from that of $\delta_i$.
    \end{rem}

       Henceforth, the symbols $\precapprox$ and $\approx$ will be used for the relations induced by the usual order $\le$ on $\mathbb N$. For the corresponding relations induced by divisibility we will, consistently with previous papers on the subject, use the notations $\tmid$ and $=_\sim$ respectively. 
        \begin{defin}\label{defin:primeuf}
          An ultrafilter is \emph{prime} iff it contains the set $\mathbb{P}$ of primes. The set of such ultrafilters is denoted by $\overline{\mathbb P}$.
        \end{defin}

    \begin{defin}
      Let $q\in S_k(\mathbb N)$ be the type of an increasing $k$-tuple of primes, that is, $q(\bla x1,k)\proves(\bla x1<k)\land\bigwedge_{i\le k} (x_i\in \mathbb P)$. We denote by $\mathcal E_{q}$ the poset of $=_\sim$-equivalence classes of ultrafilters of the form $\tp(\gamma_1^{\alpha_1}\cdot\ldots\cdot \gamma_k^{\alpha_k}/\mathbb N)$, for  $\gamma\models q$ and $\bla \alpha1,k$ some nonstandard integers.
    \end{defin}
    We emphasise that some $\gamma_i$ may generate the same prime ultrafilter, and that being twice divisible by $p$ is not the same as being divisible by $p^2$; for example, if $p=\tp(\gamma_1/\mathbb N)=\tp(\gamma_2/\mathbb N)$ but $\gamma_1\ne \gamma_2$, then $\tp(\gamma_1\cdot \gamma_2/\mathbb N)\ne\tp(\gamma_1^2/\mathbb N)$, the former is twice divisible by $p$, and the latter is divisible by $p^2$.

    \subsection{Prime powers}

 \cite[Question 6.1]{So8} asked if the orders $\mathcal E_{p}$ are isomorphic for all $p\in\overline{\mathbb{P}}\setminus\mathbb{P}$. We show in \Cref{thm:isocpindep} that the answer is independent of $\mathrm{ZFC}$, by applying to $M=\mathbb{N}(\gamma)$ the results obtained in the previous sections.

    \begin{rem}\label{rem:tmidtoprecapprox}
      In order to study $\mathcal E_p$, suppose $\alpha,\beta\in {}^*{\mathbb N}$ have a unique prime factor, and this unique prime factor has type $p$.    Since $p$ is prime, if $\alpha',\beta'$ witness that $\tp(\alpha/\mathbb N)\tmid\tp(\beta/\mathbb N)$, then there must be $\gamma\models p$ and $\eta_0\le \eta_1$ such that $\alpha'=\gamma^{\eta_0}, \beta'=\gamma^{\eta_1}$.  If furthermore $\tp(\beta/\mathbb N)\tmid\tp(\alpha/\mathbb N)$ then there must be some  $\eta_2\equiv_\gamma \eta_0$ (so, $\gamma^{\eta_2}\models\tp(\alpha/\mathbb N)$) such that $\eta_1\le \eta_2$. In other words, the map sending the $=_\sim$-class of $\tp(\gamma^{\eta}/\mathbb N)$ to the $\approx$-class of $\tp(\eta/\mathbb N(\gamma))$ is an isomorphism between $\mathcal E_p$ and $S_1(\gamma)/\mathord\approx$.
    \end{rem}
Below, whenever we mention $\operatorname{CC}$, it will always be in the sense of the order $\le$ on $\star \mathbb N$.
\begin{theorem}\label{thm:Epchar}
  For every $p\in \overline{\mathbb P}\subseteq \upbeta\mathbb N$, if $\gamma\models p$, then $\mathcal E_{p}\cong \operatorname{CC}(\gamma)$.
\end{theorem}
\begin{proof}
  By  \Cref{rem:tmidtoprecapprox,thm:approx1char}.
\end{proof}

\begin{eg}
  If  $p\in \mathbb P$ is a standard prime, then $\mathcal E_p\cong \omega+1$.
\end{eg}

  It follows immediately from \Cref{thm:Epchar} that relations in the Rudin--Keisler preorder\footnote{Recall that an ultrafilter $p$ is Rudin--Keisler-below an ultrafilter $p'$ iff there is a function $f$ such that $p=f_*p'$.} yield embeddings between different $\mathcal E_p$.
        \begin{corollary}
      \begin{enumerate}[label=(\alph*),leftmargin=*]
      \item If $p$ is Rudin--Keisler-below $p'$, then there is an embedding of $\mathcal E_p$ into $\mathcal E_{p'}$.
          \end{enumerate}
    \begin{enumerate}[label=(\alph*),resume*]
      \item If $p$ and $p'$ are Rudin--Keisler-equivalent, then $\mathcal E_p\cong \mathcal E_{p'}$. 
      \end{enumerate}
    \end{corollary}
    \begin{proof}
      \begin{enumerate}[label=(\alph*),leftmargin=*]
      \item Follows from the fact that if $p$ is Rudin--Keisler below $p'$ then $p$ is realised in the prime model over any realisation of $p'$.
          \end{enumerate}
    \begin{enumerate}[label=(\alph*),resume*]
      \item Let $f:\mathbb{N}\rightarrow\mathbb{N}$ be a bijection such that $f_*p={p'}$, which exists by~\cite[Theorem~8.17]{hindmanAlgebraStoneCechCompactification2011}. For every $\gamma\models p$, if $\delta=f(\gamma)$ then $\delta\models p'$. We have $\{g(\delta): g:\mathbb{N}\rightarrow\mathbb{N}\}=\{g(f(\gamma)): g:\mathbb{N}\rightarrow\mathbb{N}\}=\{h(\gamma): h:\mathbb{N}\rightarrow\mathbb{N}\}$, or in other words $\mathbb N(\gamma)=\mathbb N(\delta)$, and the conclusion follows.\qedhere
      \end{enumerate}
    \end{proof}

    \begin{lemma}\label{lemma:realsucc}
      If $A\subseteq \star\mathbb N$, then an element of $\operatorname{CC}(A)$ is realised in $\dcl(A)$ if and only if it has an immediate successor in $\operatorname{CC}(A)$.
    \end{lemma}
    \begin{proof}
Let $(L,R)$ be a cut. Because $\dcl(A)=\mathbb N(A)$ is a model, its order type is that of $\mathbb N$ followed by copies of $\mathbb Z$. In particular, in $\mathbb N(A)$, the only point with no immediate predecessor is $0$, and moreover it is not consistent to add points between $a$ and $a+1$. From this, it is easy to see that $(L,R)$ is realised in $\mathbb N(A)$ if and only if $R$ has a minimum $a$, if and only if $(L,R)$ has an immediate successor, namely the cut realised by $a+1$.
\end{proof}

As usual, $\mathfrak c$ denotes the cardinality of the continuum. For $p\in \beta \mathbb P$, we denote by $\mathbb N^{\mathbb P}/p$ the ultrapower of $\mathbb N$ along $p$.
    
    \begin{theorem}\label{thm:isocpindep}
      \begin{enumerate}[label=(\alph*),leftmargin=*]
      \item  If the continuum hypothesis holds, then all the $\mathcal E_p$ with $p$ nonprincipal are isomorphic.
  \end{enumerate}
    \begin{enumerate}[label=(\alph*),resume*]        
      \item $\mathrm{ZFC}+\neg \mathrm{CH}$ does not prove that all $\mathcal E_p$ with $p$ nonprincipal are isomorphic. More precisely, let $M$ be a model of $\mathrm{ZFC}+\neg \mathrm{CH}$. In every forcing extension of $M$ obtained by adding $\kappa\geq{\mathfrak c}^M$-many Cohen reals, there are $p,q\in\overline{\mathbb{P}}\setminus\mathbb{P}$ such that  $\mathcal E_p\centernot\cong \mathcal E_q$. 
      \end{enumerate}
\end{theorem}

\begin{proof}
  \begin{enumerate}[label=(\alph*),leftmargin=*]
  \item  By \Cref{thm:Epchar} it suffices to show that all $\operatorname{CC}(\gamma)$ with $\gamma\notin \mathbb N$ are isomorphic. The reduct of $\mathbb N(\gamma)$ to $\set{\le}$ is a discrete order with a minimum and no maximum, and eliminates quantifiers after naming $0$ and the successor and predecessor functions. From this it is easy to see that $\operatorname{CC}(\gamma)\cong S_1(\mathbb N(\gamma)\restr \set{\le})$. Now, the ultrapower of a countable discrete order with a minimum and no maximum over every nonprincipal ultrafilter on $\mathbb N$ is $\aleph_1$-saturated of size $\le \mathfrak c$. By standard model theoretic facts about saturated models, under $\mathrm{CH}$ there is a unique such up to isomorphism.
  \item     By \cite[Theorem, p.~95]{R} in such a forcing extension, for every regular uncountable $\lambda\leq\kappa$, there is an ultrafilter $p_\lambda$ on $\mathbb{N}$ such that $\mathbb{N}^{\mathbb{N}}/p_\lambda$ has cofinality $\lambda$.\footnote{Independently, a similar result was proven by Canjar, see \cite[Section~3.4]{C}.} By fixing a bijection between $\mathbb N$ and $\mathbb P$ we may assume that  ${p}_\lambda$ is an ultrafilter on $\mathbb{P}$. We show that $\mathcal E_{p_{\aleph_1}}\centernot\cong\mathcal E_{p_{\aleph_2}}$. Assume not. Let $I_i$ be the suborder of $\mathcal E_{p_{\aleph_i}}$ consisting of points that have an immediate successor. As we are assuming $\mathcal E_{p_{\aleph_1}}\cong \mathcal E_{p_{\aleph_2}}$, we have $I_1\cong I_2$. Let $\gamma_i\models \mathcal E_{p_{\aleph_i}}$. By \Cref{thm:Epchar}, $\mathcal E_{p_{\aleph_i}}\cong \operatorname{CC}(\gamma_i)$, and by  \Cref{lemma:realsucc} the subset of the latter consisting of points with an immediate successor is isomorphic to $\dcl(\gamma_i)=\mathbb N(\gamma_i)$, that is, by \Cref{rem:dclultrapower}, to $\mathbb{N}^{\mathbb{P}}/{p}_{\aleph_i}$. Hence, from $I_1\cong I_2$ we obtain $\mathbb{N}^{\mathbb{P}}/{p}_{\aleph_1}\cong \mathbb{N}^{\mathbb{P}}/{p}_{\aleph_2}$, contradicting that they have different cofinalities.\qedhere
  \end{enumerate}
\end{proof}

\subsection{Finitely many prime divisors}

Let us now look at the case $k\ge 2$.
\begin{rem}\label{rem:tmidtoprecapprox2}
Let $q$ be the type of an increasing $k$-tuple of primes, and  suppose that the classes of $p,p'$ are in $\mathcal E_{q}$. If $p,p'$ are comparable, there must be $\bla \gamma1,k, \bla \delta1,k, \bla \epsilon1,k$ such that $(\bla \gamma1,k)\models q$, $\alpha\coloneqq \gamma_1^{\delta_1}\cdot\ldots\cdot \gamma_k^{\delta_k}\models p$, $\beta\coloneqq \gamma_1^{\epsilon_1}\cdot\ldots\cdot \gamma_k^{\epsilon_k}\models p'$, and $\alpha\mid \beta$. It follows that, in order for $p,p'$ to be comparable, it is necessary for their squarefree parts to be equal, analogously to what we say in \Cref{rem:tmidtoprecapprox}.  By coding $\set{\bla \gamma1,k}$ with $\gamma\coloneqq \bla\gamma1\cdot k$ (or by fixing a bijection between $\mathbb N$ and $\mathbb P^{[k]}$), we therefore see that $\mathcal E_q$ is isomorphic to $(S_k(\gamma), \precapprox)/\approx$, hence we reduce to studying the latter.
\end{rem}

\begin{rem}
  Analogous considerations show that if $q$ and $q'$ are types of increasing finite tuples of primes, of possibly different length, then we have two mutually exclusive cases:
  \begin{enumerate}[label=(\alph*)]
  \item $q$ is the pushforward of $q'$ along the projection on a subset of the coordinates, or vice versa; or
  \item every point of $\mathcal E_q$ is incomparable to every point of  $\mathcal E_{q'}$.
  \end{enumerate}
\end{rem}

We now give a classification of the points of $\mathcal E_{q}$.

\begin{rem}\label{rem:speccase}
  If $(\bla \alpha1,k)<(\alpha_1',\ldots, \alpha_k')$ are realisations of $p\in S_k(\gamma)$ and $i$ is such that $\alpha_i<\alpha_i'$ then, by \Cref{rem:eqinfdist}, the point $(\alpha_1,\ldots, \alpha_{i-1}, \alpha_i+1, \alpha_{i+1},\ldots, \alpha_k)$ does not realise $p$ and lies between them. Therefore, by \Cref{rem:scfsloa} and \Cref{pr:antichsing} the $\approx$-class of $p$ is a singleton if and only if $p$ lies on an antichain.
\end{rem}

  Fix a $k$-type $q$ of an increasing sequence of primes. Let  $p$ be an ultrafilter whose $=_\sim$-class is in $\mathcal E_{q}$. Fix
  $\alpha\models p$ of the form $\alpha=\prod_{i\le k}\gamma_i^{\delta_i}$, where $(\bla \gamma1,k)\models q$, and let $\gamma\coloneqq  \bla \gamma 1\cdot k$. Then, we have five mutually exclusive cases.
  \begin{enumerate}[label=(\alph*)]
  \item\label{point:lowerdim} All $\delta_i$ lie in $\mathbb N(\gamma)$, so there are functions $f_i\from \mathbb N\to \mathbb N$ such that $\alpha=\prod_{i\le k}\gamma_i^{f_i(\gamma)}$. Clearly, such a $=_\sim$-class must be a singleton.
  \item \label{point:nrac} No $\delta_i$ lies in $\mathbb N(\gamma)$ and, for every $i\le k$, there is a strictly decreasing, definable partial function $g_i\from \mathbb N(\gamma)^{k-1}\to \mathbb N(\gamma)$ such that $\delta_i=g_i(\delta_1,\ldots, \delta_{i-1},\delta_{i+1},\ldots,\delta_k)$. By  \Cref{cor:liesanti} and  \Cref{rem:speccase}, the class of $p$ is a singleton.
  \item\label{point:fulldim} There is a product $\bla I1\times k$ of infinite intervals containing $(\bla \delta1,k)$ such that all $I_i\cap \set{x<\delta_i}$ and all $I_i\cap \set{x>\delta_i}$ are infinite and such that,  for every  $(\bla \epsilon1,k)\in \bla I1\times k$, the type of $\prod_{i\le k} \gamma_i^{\epsilon_i}$ is $=_{\sim}$-equivalent to $p$.
  \item\label{point:ladder} No $\delta_i$ lies in $\mathbb N(\gamma)$, the class of $p$ is not a singleton, and we are not in case~\ref{point:fulldim}. Then, by \Cref{pr:box}, whenever $\delta'\equiv_\gamma \delta$ is such that $\delta\ge \delta'$, there is $i\le k$ such that $\delta_i=\delta_i'$.
  \item\label{point:mixed} For some $\ell>0$ there are indices, say $i=\ell+1,\ell+2,\ldots,k$ (up to a permutation of the variables), such that $\delta_i\in\mathbb N(\gamma)$, and for $j\le \ell$ we have $\delta_j\notin \mathbb N(\gamma)$. Then the class of $p$ naturally corresponds to some class in $\mathcal E_{q'}$ of kind \ref{point:nrac}, \ref{point:fulldim}, or \ref{point:ladder}, where $q'$ is the projection of $q$ on the first $\ell$ coordinates.
  \end{enumerate}

\begin{eg}
 Examples of each of these cases may be obtained as follows.
 \begin{enumerate}[label=(\alph*)]
 \item Consider, e.g., $\gamma_1^{\gamma_1}\cdot\ldots\cdot \gamma_k^{\gamma_k}$.
 \item Take some pairwise distinct  $\delta_1,\ldots, \delta_{k-1}$ not in $\mathbb N(\gamma)$ and all in Archimedean classes smaller than that of $\gamma_1$,  and set $g_k(\bla x1,{k-1})= \gamma_1-(\bla x1+{k-1})$. 
 \item This happens for instance if the tuple $(\bla \gamma1,k,\bla \delta1,k)$ is tensor with all coordinates infinite (see \cite{L} for the definition of, and various results on, tensor tuples).
 \item Let $k=2$, and take $\delta_1> \mathbb N$ such that the pair $(\delta_1, \gamma)$ is tensor. Consider the interval $[\gamma\cdot (\gamma-\delta_1), \gamma\cdot (\gamma-\delta_1+1)]$. Take an arbitrary $\delta_2$ in this interval but not in $\mathbb N(\set{\delta_1, \gamma})$.  The type of $\gamma_1^{\delta_1}\gamma_2^{\delta_2}$ does not lie on an antichain, as we may take any $\delta_2'\equiv_{\delta_1,\gamma} \delta_2$ and consider $\gamma_1^{\delta_1}\gamma_2^{\delta_2'}$. On the other hand, whenever $(\delta_1'', \delta_2'')\equiv_{\gamma}(\delta_1,\delta_2)$ and $\delta_1''> \delta_1$, by \Cref{rem:eqinfdist} we also have $\delta_1''> \delta_1+1$, and it follows that $\delta_2''<\delta_2$.
 \item Set $\bla \delta{\ell+1}=k=1$ and choose $\bla \delta1,\ell$ in such a way that $\gamma_1^{\delta_1}\cdot\ldots\cdot \gamma_\ell^{\delta_\ell}$ is of kind \ref{point:nrac}, \ref{point:fulldim}, or \ref{point:ladder}.
 \end{enumerate}
\end{eg}

\begin{rem}
  \begin{enumerate}  [label=(\alph*),leftmargin=*]
  \item   It follows by inspection of the cases above that, if $\delta,\epsilon\in\star \mathbb N^k$ are at finite distance, then the classes of $\tp(\prod_{i\le k} \gamma_i^{\delta_i}/\mathbb N)$ and of  $\tp(\prod_{i\le k} \gamma_i^{\epsilon_i}/\mathbb N)$ are of the same kind.

  \item Taking supremums may result in a change of kind.   For example, let $u_n\coloneqq\tp(\gamma_1^{1}, \gamma_2^n)$. For all $n$, the class of $u_n$ is of kind~\ref{point:lowerdim}, but the supremum of these classes is of kind~\ref{point:mixed}.
  \end{enumerate}
\end{rem}

\section{Further directions}\label{sec:croq}
In \Cref{thm:approx1char} we gave a fairly explicit description of $S_1(A)/\mathord\approx$. The problem of finding a nice characterisation in the higher dimensional case still remains.
    \begin{problem}\label{prob:descska}
For $A$ a subset of a linearly ordered, definably complete structure, describe the posets $S_k(A)/\mathord\approx$.
    \end{problem}
    As for divisibility, even assuming a satisfactory answer to \Cref{prob:descska}, it remains open to give a precise description of how the various $\mathcal E_q$ fit together.

    Furthermore, the case of ultrafilters divisible by infinitely many primes is still largely unexplored.  Still, some results and techniques from this paper apply also to the case of infinitely many prime divisors; for instance, one can easily adapt the argument from \Cref{rem:speccase} and obtain that an ultrafilter lies on an antichain if and only if it has a singleton $=_\sim$-equivalence class.

    A tool introduced to study this problem are so-called patterns, developed in \cite{So8} and \cite{So10}. Patterns reflect the quantity of each element of $\mathcal E_p$ (for various $p\in \overline{\mathbb P}$) dividing the given ultrafilter. They do not determine completely the place of an ultrafilter in the divisibility preorder, nor are determined by it, but several divisibility-related properties have equivalent versions in terms of patterns.

     Observe that the study of divisibility has an equivalent formulation.

    \begin{rem}
      There is an isomorphism between $(\mathbb N, \mid)$ and the poset of finite multisets of natural numbers with inclusion.
    \end{rem}

\begin{problem}
  Describe the order on ultrafilters on finite multisets of natural numbers induced by inclusion.
\end{problem}

A first natural subcase is that of finite sets, that is, the study of the relation $\precapprox$ induced by the partial order $\subseteq$ on  $\mathcal P_{<\omega}(\mathbb N)$, where the latter is equipped with the full language, namely  with a symbol for every possible relation. This corresponds to divisibility amongst squarefree natural numbers. Even in this special case, one cannot argue anymore by fixing primes $\gamma_i$ as in \Cref{sec:div}, as visible in the following examples.

\begin{eg}
Let $\pi$ be the increasing enumeration of the primes and let $\alpha> \mathbb N$. By \Cref{pr:antichsing}, the type of $\prod_{i=\alpha}^{2\alpha} \pi(i)$ forms a non-realised singleton $=_{\sim}$-class, since it contains the antichain $\{\prod_{i=n}^{2n} \pi(i): n\in\mathbb{N}\}$.
\end{eg}

\begin{eg}With the same notation as in the previous example, all ultrafilters generated by $\prod_{i=1}^{\alpha} \pi(i)$, with  $\alpha>\mathbb N$, are in the same $=_{\sim}$-class. In fact, for every $\alpha_1<\alpha_2$, there is $\alpha_3>\alpha_2$ with the same type as $\alpha_1$; it follows that $\prod_{i=1}^{\alpha_1}\pi(i)$ and $\prod_{i=1}^{\alpha_3}\pi(i)$ have the same type, and clearly $\prod_{i=1}^{\alpha_1} \pi(i)\mid\prod_{i=1}^{\alpha_2} \pi(i)\mid\prod_{i=1}^{\alpha_3} \pi(i)$. 
\end{eg}

Let us conclude with a final example, phrased directly in terms of inclusion on hyperfinite sets. It shows that there are intervals with both endpoints not in $\mathbb N$ that do not form singleton classes.

\begin{eg}
 For fixed $\delta$, all intervals of the form $(\delta-\gamma, \delta+\gamma)$ with $(\gamma,\delta)$ tensor are in the same $\approx$-class: to witness the inequality $(\delta-\epsilon, \delta+\epsilon)\precapprox(\delta-\gamma, \delta+\gamma)$ for $\epsilon>\gamma$, just replace $\epsilon$ by some $\epsilon'$ such that the triple $(\epsilon',\gamma,\delta)$ is tensor.
\end{eg}

\newcommand{\noop}[1]{}

\footnotesize

\newcommand{\etalchar}[1]{$^{#1}$}


\begin{thebibliography}{DNLBM{\etalchar{+}}23}

\bibitem[Can88]{C}
Michael Canjar.
\newblock Countable ultraproducts without {{CH}}.
\newblock {\em Ann. Pure Appl. Logic}, 37(1):1--79, 1988.

\bibitem[DLMPR25]{dinassoSelfdivisibleUltrafiltersCongruences2023a}
Mauro Di~Nasso, Lorenzo Luperi~Baglini, Rosario Mennuni, Moreno Pierobon, and
  Mariaclara Ragosta.
\newblock Self-divisible ultrafilters and congruences in {{$\beta \mathbb Z$}}.
\newblock {\em J. Symb. Log.}, 90(3):1180--1197, 2023.
\newblock {doi:
  \href{https://doi.org/10.1017/jsl.2023.51}{\texttt{10.1017/jsl.2023.51}}}.

\bibitem[FS10]{fornasieroDefinablyCompleteBaire2010}
Antongiulio Fornasiero and Tamara Servi.
\newblock Definably complete {{Baire}} structures.
\newblock {\em Fundam. Math.}, 209(3):215--241, 2010.

\bibitem[Hie13]{hieronymiAnalogueBaireCategory2013}
Philipp Hieronymi.
\newblock An analogue of the {{Baire}} category theorem.
\newblock {\em J. Symb. Log.}, 78(1):207--213, 2013.

\bibitem[Hru19]{hrushovskiDefinabilityPatternsTheir2020}
Ehud Hrushovski.
\newblock Definability patterns and their symmetries.
\newblock Preprint,
  \href{https://arxiv.org/abs/1911.01129}{\texttt{arxiv.org/abs/1911.01129}},
  2019.

\bibitem[HS11]{hindmanAlgebraStoneCechCompactification2011}
Neil Hindman and Dona Strauss.
\newblock {\em Algebra in the {{Stone-\v Cech Compactification}}: {{Theory}}
  and {{Applications}}}.
\newblock De Gruyter, 2011.

\bibitem[LB19]{L}
Lorenzo Luperi~Baglini.
\newblock Nonstandard characterisations of tensor products and monads in the
  theory of ultrafilters.
\newblock {\em MLQ Math. Log. Q.}, 65(3):347--369, 2019.

\bibitem[Mil01]{millerExpansionsDenseLinear2001}
Chris Miller.
\newblock Expansions of dense linear orders with the intermediate value
  property.
\newblock {\em J. Symb. Log.}, 66(4):1783--1790, 2001.

\bibitem[PS21]{poliakovUltrafilterExtensionsFirstorder2021}
Nikolai~L. Poliakov and Denis~I. Saveliev.
\newblock On ultrafilter extensions of first-order models and ultrafilter
  interpretations.
\newblock {\em Arch. Math. Logic}, 60(5):625--681, 2021.

\bibitem[Roi82]{R}
Judy Roitman.
\newblock Non-isomorphic hyper-real fields from non-isomorphic ultrapowers.
\newblock {\em Math. Z.}, 181(1):93--96, 1982.

\bibitem[Sav15]{savelievUltrafilterExtensionsLinearly2015}
Denis~I. Saveliev.
\newblock Ultrafilter {{Extensions}} of {{Linearly Ordered Sets}}.
\newblock {\em Order}, 32(1):29--41, 2015.

\bibitem[{\v S}ob21]{sobotCONGRUENCEULTRAFILTERS2021}
Boris {\v S}obot.
\newblock Congruence of ultrafilters.
\newblock {\em J. Symb. Log.}, 86(2):746--761, 2021.

\bibitem[{\v S}ob25a]{So8}
Boris {\v S}obot.
\newblock \noop{a}$\widetilde{\mid}$-divisibility of ultrafilters {{II}}:
  {{The}} big picture.
\newblock {\em Rep. Math. Logic}, (60), 2025.
\newblock To appear, preprint available at
  \href{https://arxiv.org/abs/2306.00101}{\texttt{arxiv.org/abs/2306.00101}}.

\bibitem[{\v S}ob25b]{So10}
Boris {\v S}obot.
\newblock \noop{b} {Divisibility} classes of ultrafilters and their patterns.
\newblock Preprint,
  \href{https://arxiv.org/abs/2412.19753}{\texttt{arxiv.org/abs/2412.19753}},
  2025.

\end{thebibliography}
\end{document}